\documentclass[12pt,reqno]{article}

\usepackage[T1]{fontenc}
\usepackage[utf8]{inputenc}
\usepackage{lmodern}
\usepackage{microtype}

\usepackage{amsmath,amssymb,amsthm}
\usepackage{hyperref}

\hypersetup{
  colorlinks=true,
  linkcolor=blue,
  citecolor=blue,
  urlcolor=blue
}

\newtheorem*{maintheorem}{Main theorem}
\newtheorem*{maincorollary}{Corollary}
\newtheorem{remark}{Remark}
\newtheorem{definition}{Definition}

\newtheorem{example}{Example}
\newtheorem{proposition}{Proposition}
\newtheorem{theorem}{Theorem}
\newtheorem{corollary}{Corollary}
\newtheorem{lemma}{Lemma}

\title{On finite perfect two-sided skew braces}
\author{Marco Damele\thanks{Dipartimento di Matematica, Università di Cagliari,
Via Ospedale 72, 09124 Cagliari, Italy;
\texttt{marco.damele@unica.it}; ORCID 0009-0008-3088-5766.
The author was supported by INdAM and GNSAGA -- Gruppo Nazionale per le
Strutture Algebriche, Geometriche e le loro Applicazioni, and by ProBiki,
funded by Fondazione di Sardegna.}}
\date{}

\begin{document}

\maketitle

\begin{abstract}
We prove a structure theorem for finite perfect two-sided skew braces. The main
tool is a central product theory for skew braces, developed here in both
external and internal form; we show that these two constructions are
equivalent. Our main result states that every finite perfect two-sided skew
brace \(B\) admits the canonical decomposition
$B=B^2\circ B^{2,\operatorname{op}},$
where \(B^2\) is almost trivial with perfect additive group, while
\(B^{2,\operatorname{op}}\) is trivial with perfect additive group. Thus finite
perfect two-sided skew braces are classified, up to central amalgamation, by
trivial and almost trivial skew braces arising from perfect groups. This
decomposition has strong consequences for the underlying groups: for finite
two-sided skew braces, perfectness of the skew brace is equivalent to
perfectness of either the additive or the multiplicative group. In the
trivial-center case the central product becomes a direct product, recovering
Trappeniers' classification of finite simple two-sided skew braces. We also
show that quasi-simple two-sided skew braces are necessarily either trivial or
almost trivial. Finally, we prove that this rigidity is genuinely two-sided by
constructing a quasi-simple skew brace which is not two-sided and is neither
trivial nor almost trivial.
\end{abstract}



\newpage
\tableofcontents

\section{Introduction}

Skew braces were introduced by Guarnieri and Vendramin in
\cite{GuarnieriVendramin2017} as a non-abelian generalization of left braces
introduced by Rump in \cite{Rump2007}, with the aim of providing an algebraic
framework for the study of set-theoretic solutions of the Yang--Baxter
equation. Since then, skew braces have become an important meeting point between
group theory, ring theory and the theory of the Yang--Baxter equation. 

Formally, a skew brace is a triple \((B,+,\cdot)\) such that \((B,+)\) and
\((B,\cdot)\) are groups and, for all \(a,b,c\in B\), the following
compatibility condition holds:
\[
a\cdot(b+c)=a\cdot b-a+a\cdot c.
\]
\emph{Throughout the paper, all skew braces are assumed to be finite}.
A recurring theme in the theory of skew braces is the interaction between the
two underlying groups. One is often interested in understanding how structural
properties of the additive group \((B,+)\), or of the multiplicative group
\((B,\cdot)\), constrain the skew brace structure, and conversely how
brace-theoretic properties are reflected in the two group structures. This
perspective has led to several structural results, for example in the study of
skew braces with prescribed underlying groups and in the analysis of nilpotency,
solubility and related finiteness conditions; see, for instance,
\cite{Damele2026, Trappeniers2023, BallesterBolinchesEstebanRomeroJimenezSeralPerezCalabuig2024}.

The present paper follows this line of investigation in the case of
perfectness. Since perfectness is naturally expressed in terms of quotients, it
is useful to recall that, in skew brace theory, quotients are taken with respect
to ideals, the analogues of normal subgroups. The precise definition will be
given in Section~\ref{Preliminaries}.

In group theory, perfect groups form a fundamental class: they are
the groups with no non-trivial abelian quotients. For skew braces there is more
than one possible notion of perfectness. In \cite{CSV2019}, a skew brace \(B\)
is called perfect if \(B=B^2\), where \(B^2=B*B\) is the ideal generated by the
brace products; see Section~\ref{Preliminaries} for the relevant definitions.
This condition measures perfectness through the brace product alone. In the
present paper we use a different, and more group-theoretic, notion: we say that
\(B\) is perfect if $B=[B,B]^B,$
where \([B,B]^B\) denotes the commutator ideal of \(B\), namely the smallest
ideal such that the corresponding quotient is an abelian skew brace. Thus, in
this sense, a perfect skew brace is precisely a skew brace with no non-trivial
abelian quotients. In particular, every simple non-abelian skew brace is perfect. 

We specialize our study to the two-sided case. Recall that a skew brace \(B\) is
called \emph{two-sided} if it also satisfies the right brace identity
\[
(a+b)c=ac-c+bc
\]
for all \(a,b,c\in B\). For example, every group \(G\) gives rise to the trivial two-sided skew brace $\operatorname{Triv}(G)=(G,\cdot,\cdot),$
in which both operations coincide with the group operation of \(G\). We shall
also use the almost trivial skew brace associated with \(G\), denoted by $\operatorname{aTriv}(G)=(G,\cdot^{\operatorname{op}},\cdot),$
where \(\cdot^{\operatorname{op}}\) is the opposite operation of \(\cdot\).

Two-sided skew braces form a distinguished and particularly rigid subclass of
skew braces. In the case where the additive group is abelian, they correspond
to Jacobson radical rings; see \cite{CV}. In the general non-abelian setting,
they still retain strong ring-theoretic features and provide a natural
framework for studying the interaction between the additive and multiplicative
groups; see \cite{Nasybullov2019,Trappeniers2023}. The importance of this class
is also reflected in its connection with a well-known open problem in the
theory of skew braces: whether, for a finite skew brace, solubility of the
additive group implies solubility of the multiplicative group. While this
problem remains open in general, Nasybullov proved that it has an affirmative
answer for two-sided skew braces \cite{Nasybullov2019}. Thus the two-sided
identity imposes strong constraints on the relation between the two underlying
groups, making this class especially suitable for structural investigations.

The main goal of this paper is to understand how perfectness behaves in the
two-sided setting. More precisely, we ask to what extent the condition
\(B=[B,B]^B\) forces structural properties on the additive group \((B,+)\) and
on the multiplicative group \((B,\cdot)\), and conversely how perfectness of
the underlying groups is reflected in the brace structure. Answering these
questions requires a detailed structural description of finite perfect
two-sided skew braces.

Finite simple two-sided skew braces were classified by Trappeniers
\cite{Trappeniers2023}. In particular, his classification shows that a finite
simple two-sided skew brace is necessarily one of the two standard types: it is
either trivial or almost trivial. Equivalently, it is isomorphic to
\(\operatorname{Triv}(G)\) or to \(\operatorname{aTriv}(G)\), where \(G\) is a
finite non-abelian simple group.

The present paper extends this picture from the simple case to the broader
class of finite perfect two-sided skew braces. The passage from simple to
perfect braces introduces a new phenomenon. Indeed, a perfect two-sided skew
brace need not be simple and may have non-trivial centre. Hence one should not
expect such a brace to be described by a single trivial or almost trivial
component. Instead, the natural object replacing the two alternatives in the
simple case is a decomposition, in which the two canonical components
\(B^2\) and \(B^{2,\operatorname{op}}\) are glued together along a common
central ideal. Here \(B^{2,\operatorname{op}}\) denotes the square of the opposite skew brace,
namely \(B^{2,\operatorname{op}}=(B^{\operatorname{op}})^2\); see
Section~\ref{Preliminaries} for the relevant definitions. This is the point of view developed in the present paper. We first introduce a
central product theory for skew braces, both in external and internal form, and
prove that the two constructions are equivalent. With this terminology, our main result says that finite perfect two-sided skew
braces are described, up to central amalgamation, by trivial and almost trivial
skew braces associated with perfect groups.

\begin{maintheorem}
Let \(B\) be a finite perfect two-sided skew brace. Then:
\begin{enumerate}
    \item \(B=B^{2} \circ B^{2,\operatorname{op}}\);
    \item \(B^2\) is almost trivial with perfect additive group;
    \item \(B^{2,\operatorname{op}}\) is trivial with perfect additive group.
\end{enumerate}
\end{maintheorem}

This decomposition has several consequences. First, it shows that, in the
two-sided setting, the brace-theoretic notion of perfectness is completely
detected by either of the two underlying groups.

\begin{maincorollary}
Let \(B\) be a finite two-sided skew brace. Then the following conditions are
equivalent:
\begin{enumerate}
    \item \(B\) is perfect;
    \item the additive group \((B,+)\) is perfect;
    \item the multiplicative group \((B,\cdot)\) is perfect.
\end{enumerate}
\end{maincorollary}

This equivalence shows that finite two-sided skew braces behave, with respect
to perfectness, much more rigidly than arbitrary skew braces. The same
decomposition also gives a transparent description in the case of trivial
center.

\begin{maincorollary}
Let \(B\) be a finite perfect two-sided skew brace. If \(Z(B)=0\), then
\[
B \simeq \operatorname{Triv}(G) \times \operatorname{aTriv}(H)
\]
for some finite perfect groups \(G\) and \(H\).
\end{maincorollary}

In particular, since every simple skew brace has trivial center, this recovers
Trappeniers' classification of finite simple two-sided skew braces as a special
case. Moreover, using the fact that finite perfect two-sided skew braces
satisfy
\[
Z(B/Z(B))=0,
\]
see \cite{Tsang2026}, we obtain the following description of the quotient by
the center.

\begin{maincorollary}
Let \(B\) be a finite perfect two-sided skew brace. Then
\[
B/Z(B)\simeq \operatorname{Triv}(G)\times \operatorname{aTriv}(H)
\]
for some finite perfect groups \(G\) and \(H\).
\end{maincorollary}

We also apply the structural theorem to a skew-brace analogue of quasi-simple
groups. We call a skew brace \(B\) \emph{quasi-simple} if \(B\) is perfect and
\(B/Z(B)\) is simple. In the two-sided case, the above decomposition imposes a
strong restriction: a quasi-simple skew brace cannot be obtained as a genuine
central product of a non-trivial trivial component and a non-trivial almost
trivial component. More precisely, we prove the following result.

\begin{maincorollary}
Let \(B\) be a finite two-sided skew brace. If one of the following conditions
holds:
\begin{enumerate}
    \item \(B\) is quasi-simple;
    \item the additive group \((B,+)\) is quasi-simple;
    \item the multiplicative group \((B,\cdot)\) is quasi-simple,
\end{enumerate}
then \(B\) is either trivial or almost trivial.
\end{maincorollary}

Thus quasi-simplicity rules out genuine mixed central products in the two-sided
setting. Consequently, every finite quasi-simple two-sided skew brace comes
from a quasi-simple group in one of only two ways: either as a trivial skew
brace or as an almost trivial skew brace.

Finally, we prove that the two-sided assumption cannot be removed. More
precisely, we construct a quasi-simple skew brace which is neither trivial nor
almost trivial. The construction uses exact factorizations of groups and the
associated skew braces described in
\cite[Theorem~2.3]{SmoktunowiczVendramin2018}. Starting from the exact
factorization
\[
A_5=A_4C_5,
\]
we obtain a simple skew brace and then lift this construction to
\(\operatorname{SL}(2,5)\). The resulting skew brace has additive group
isomorphic to \(\operatorname{SL}(2,5)\), is quasi-simple, but is not
two-sided. This shows that the rigidity phenomenon for quasi-simple
two-sided skew braces is genuinely two-sided and does not extend to arbitrary
skew braces.

The paper is organized as follows. In Section~\ref{Preliminaries} we recall the
basic definitions and notation on skew braces that will be used throughout the
paper. In Section~\ref{central product of skew braces} we introduce external
and internal central products of skew braces and prove that the two
constructions are equivalent. We also discuss examples illustrating the
construction. In Section~\ref{Perfect two-sided skew braces} we study finite
perfect two-sided skew braces and prove the main structural decomposition
\(B=B^2\circ B^{2,\operatorname{op}}\), together with its consequences for the
additive and multiplicative groups. We then apply this decomposition to
quasi-simple two-sided skew braces, proving that they are necessarily either
trivial or almost trivial. The final part of the paper shows that the
two-sidedness assumption is essential by constructing a quasi-simple skew brace
which is not two-sided and is neither trivial nor almost trivial.
\section{Preliminaries} \label{Preliminaries}

We collect in this section the basic definitions and notation on skew braces
needed in the sequel. For a more detailed introduction to skew braces, we refer
the reader to \cite{GuarnieriVendramin2017} and
\cite{SmoktunowiczVendramin2018}, which also contains an appendix by Byott and
Vendramin.

Let \(B\) be a skew brace. We shall use additive notation for the group
\((B,+)\) and multiplicative notation for the group \((B,\cdot)\). The
identity element of the two groups is the same and will be denoted by
\(0\). For \(b\in B\), define
\[
\lambda_b:B\longrightarrow B,\qquad
\lambda_b(a)=-b+ba.
\]
Then \(\lambda_b\in \operatorname{Aut}(B,+)\), and the map
\[
\lambda:(B,\cdot)\longrightarrow \operatorname{Aut}(B,+),
\qquad
b\longmapsto \lambda_b,
\]
is a group homomorphism. A subset \(I\subseteq B\) is called an \emph{ideal} of \(B\), and we
write \(I\trianglelefteq B\), if \((I,+)\trianglelefteq (B,+)\),
\((I,\cdot)\trianglelefteq (B,\cdot)\), and
$\lambda_b(I)\subseteq I$ for every \(b\in B\). In this case one can form the quotient skew brace \((B/I,+,\cdot)\), where additive and multiplicative cosets coincide, that is, $a+I=aI$ for every $a \in B$. If \(I\trianglelefteq B\), we denote by \([B:I]\) the index of the subgroup
\((I,+)\) in the additive group \((B,+)\). The \emph{socle} of \(B\) is
\[
\operatorname{Soc}(B)=\ker(\lambda)\cap Z(B,+).
\]
It is an ideal of \(B\). The \emph{center} of \(B\) is defined by
\[
Z(B)=\operatorname{Soc}(B)\cap Z(B,\cdot).
\]
Thus \(z\in Z(B)\) if and only if \(z\) is central in both underlying
groups and \(\lambda_z=\operatorname{id}_B\). The center \(Z(B)\) is an ideal of \(B\); in some papers it is called the
\emph{annihilator} of \(B\) and denoted by \(\operatorname{Ann}(B)\).  Let \(A\) and \(B\) be skew braces. A map \(f:A\to B\) is a
\emph{homomorphism of skew braces} if
\[
f(a+a')=f(a)+f(a')
\qquad\text{and}\qquad
f(aa')=f(a)f(a')
\]
for all \(a,a'\in A\). A bijective homomorphism is called an
isomorphism. If \(A\) and \(B\) are isomorphic, we write $A\simeq B.$ We denote by \(A\times B\) the direct product skew brace, that is, the
skew brace whose underlying set is the Cartesian product \(A\times B\)
and whose operations are defined componentwise.
For \(a,b\in B\), set
\[
a*b=-a+ab-b.
\]
If \(X,Y\subseteq B\), we denote by \(X*Y\) the additive subgroup of
\((B,+)\) generated by all elements
\[
x*y,
\qquad x\in X,\ y\in Y.
\]
We write $B^2=B*B.$  We say that \(B\) is
\emph{trivial} if $B^2=0.$
Equivalently, \(B\) is trivial if the two operations coincide, that is, $ab=a+b$
for all \(a,b\in B\). It is well known that \(B^2\) is the smallest ideal of \(B\) such that $B/B^{2}$ is trivial.
A skew brace \(B\) is called a \emph{left brace} if its additive group
\((B,+)\) is abelian. We say that \(B\) is \emph{abelian} if it is a
left brace and \(B^2=0\). Equivalently, \(B\) is abelian if both
underlying groups are abelian and the two operations coincide. If \(G\) is a group, then
$\operatorname{Triv}(G)=(G,\cdot,\cdot)$
denotes the trivial skew brace over \(G\). Here both the additive and
multiplicative operations are equal to the group operation of \(G\).  We shall also use the opposite skew brace. It is defined by
\[
B^{\operatorname{op}}=(B,+^{\operatorname{op}},\cdot),
\]
where \(+^{\operatorname{op}}\) is the opposite operation of \(+\), that
is, $a+^{\operatorname{op}} b=b+a.$
We set $B^{2,\operatorname{op}}=(B^{\operatorname{op}})^2.$
A skew brace \(B\) is called \emph{almost trivial} if $B^{2,\operatorname{op}}=0.$
Equivalently, \(B\) is almost trivial if
$ab=b+a$
for all \(a,b\in B\). We put $\operatorname{aTriv}(G)=\operatorname{Triv}(G)^{\operatorname{op}}.$

\section{Central products of skew braces} \label{central product of skew braces}

As explained in the introduction, central products will be the main tool for
describing the structure of finite perfect two-sided skew braces. The main
structural theorem of the paper will show that every such skew brace decomposes
as a central product of two canonical ideals. For this reason, we first develop
the notion of central product in the general setting of skew braces, in both an
external and an internal form, and then prove that the two constructions are
equivalent.

\subsection{External and internal central product}

Let $B$ be a skew brace. An ideal \(I\) of \(B\) will be called \emph{central} if \(I\leq Z(B)\). Notice that, if
\(I\leq Z(B)\), then the two operations of \(B\) coincide on \(I\), and \(I\) is a trivial
skew brace. The following definition is inspired by the classical notion of central
product in group theory; see, for instance, \cite{Suzuki1982}.

\begin{definition}[External central product] \label{external central product definition} 
Let \(B_1\) and \(B_2\) be skew braces. Let $I_1\leq Z(B_1)$ and $I_2\leq Z(B_2)$
be central ideals, and let $\alpha:I_1\longrightarrow I_2$
be an isomorphism of skew braces. The \emph{external central product} of \(B_1\) and
\(B_2\) amalgamating \(I_1\) and \(I_2\) through \(\alpha\) is the skew brace
\[
B_1\circ_\alpha B_2
:=
\frac{B_1\times B_2}{N_\alpha},
\]
where
\[
N_\alpha=\{(i,-\alpha(i))\mid i\in I_1\}.
\]
\end{definition}

\begin{remark} \rm
 Since \(I_1\leq Z(B_1)\) and \(I_2\leq Z(B_2)\), the subgroup \(N_\alpha\) is a
central ideal of \(B_1\times B_2\). Thus the quotient above is again a skew brace.  
\end{remark}

\begin{remark} \rm
The external central product \(B_1\circ_\alpha B_2\) can be viewed as a
direct product in which the central ideals \(I_1\) and \(I_2\) are
identified through the isomorphism \(\alpha\). The minus sign in the
definition of \(N_\alpha\) is precisely what imposes this identification.
Indeed, for every \(i\in I_1\), we have
\[
(i,-\alpha(i))\in N_\alpha.
\]
Hence, in the quotient \((B_1\times B_2)/N_\alpha\),
\[
(i,0)+N_\alpha=(0,\alpha(i))+N_\alpha,
\]
because
\[
(i,0)-(0,\alpha(i))=(i,-\alpha(i))\in N_\alpha.
\]
Thus each element \(i\in I_1\) is identified with its image
\(\alpha(i)\in I_2\).
\end{remark}

The following example shows that the central product of skew braces generalizes the well known central product of groups.

\begin{example}[Central products of trivial skew braces] \rm
Let \(G_1\) and \(G_2\) be groups, and let $Z_1\leq Z(G_1)$ and $Z_2\leq Z(G_2)$ be central subgroups. Let $\alpha:Z_1\longrightarrow Z_2$
be an isomorphism of groups. Then 
\[
\operatorname{Triv}(G_1)\circ_\alpha \operatorname{Triv}(G_2) = \operatorname{Triv}((G_1\times G_2)/\{(z,\alpha(z)^{-1})\mid z\in Z_1\})
\]
Thus the notion of central product for skew braces extends the classical central product
of groups.
\end{example}

\begin{definition}[Internal central product] 
Let \(B\) be a skew brace, and let \(I,J\) be ideals of \(B\). We say that \(B\) is the
\emph{internal central product} of \(I\) and \(J\) if
\begin{enumerate}
    \item $B=I+J$;
    \item $[I,J]_{+}=0$;
    \item $I*J=J*I=0$;
    \item $I \cap J \le Z(B)$
\end{enumerate}
In this case we write $B=I\circ J.$
\end{definition}

The next proposition shows that an external central product can be realized as
an internal central product of two naturally embedded ideals. The argument is
formally analogous to the group-theoretic one; we include the proof for
completeness.

\begin{proposition}[External central products are internal]
Let \(B_1\circ_\alpha B_2\) be the external central product of \(B_1\) and
\(B_2\), obtained by amalgamating the central ideals \(I_1\leq Z(B_1)\) and
\(I_2\leq Z(B_2)\) through the isomorphism
$\alpha:I_1\longrightarrow I_2.$ Let
\[
\overline B_1=(B_1\times 0+N_\alpha)/N_\alpha
\qquad\text{and}\qquad
\overline B_2=(0\times B_2+N_\alpha)/N_\alpha.
\]
Then
\[
B_1\circ_\alpha B_2=\overline B_1\circ \overline B_2
\]
and
\[
\overline B_1\cap \overline B_2\simeq I_1\simeq I_2.
\]
\end{proposition}

\begin{proof}
Put \(P=(B_1\times B_2)/N_\alpha\), where
\(N_\alpha=\{(i,-\alpha(i))\mid i\in I_1\}\), and let
\(\pi:B_1\times B_2\to P\) be the canonical projection. Then
\[
\overline B_1=\pi(B_1\times 0),
\qquad
\overline B_2=\pi(0\times B_2).
\]
Since \(B_1\times 0\) and \(0\times B_2\) are ideals of \(B_1\times B_2\),
their images \(\overline B_1\) and \(\overline B_2\) are ideals of \(P\).
Moreover, every element of \(P\) can be written as
\[
\pi(b_1,b_2)=\pi(b_1,0)+\pi(0,b_2),
\]
and therefore \(P=\overline B_1+\overline B_2\).
The mixed additive commutators and mixed brace products vanish because this is
already true in the direct product \(B_1\times B_2\). Indeed, for
\(b_1\in B_1\) and \(b_2\in B_2\), one has
\[
[(b_1,0),(0,b_2)]_+=0,
\qquad
(b_1,0)*(0,b_2)=0=(0,b_2)*(b_1,0).
\]
Applying the quotient map \(\pi\), we get
\[
[\overline B_1,\overline B_2]_+=0,
\qquad
\overline B_1*\overline B_2=0=\overline B_2*\overline B_1.
\]
It remains to describe the intersection. Let
\(x\in \overline B_1\cap \overline B_2\). Then
\(x=\pi(b_1,0)=\pi(0,b_2)\) for some \(b_1\in B_1\) and \(b_2\in B_2\).
Equivalently, \((b_1,-b_2)\in N_\alpha\). Hence there exists \(i\in I_1\)
such that \((b_1,-b_2)=(i,-\alpha(i))\). Thus \(b_1=i\) and
\(b_2=\alpha(i)\), and so
\[
\overline B_1\cap \overline B_2
=
\{\pi(i,0)\mid i\in I_1\}
=
\{\pi(0,j)\mid j\in I_2\}.
\]
The map \(I_1\to \overline B_1\cap \overline B_2\), \(i\mapsto \pi(i,0)\), is
an isomorphism. Hence
\[
\overline B_1\cap \overline B_2\simeq I_1\simeq I_2.
\]
Finally, this intersection is central in \(P\). Indeed, every element of
\(\overline B_1\cap \overline B_2\) has the form \(\pi(i,0)\) with
\(i\in I_1\). Since \(I_1\leq Z(B_1)\), the element \((i,0)\) is central in
\(B_1\times B_2\), and therefore \(\pi(i,0)\in Z(P)\). Thus
\(\overline B_1\cap \overline B_2\leq Z(P)\).
 Therefore
\(P=\overline B_1\circ \overline B_2\).
\end{proof}

\begin{proposition}[Internal central products are external]
Let \(B=I\circ J\) be an internal central product of two ideals \(I\) and \(J\).
Then \(B\) is isomorphic to the external central product of \(I\) and \(J\)
obtained by amalgamating the common central ideal \(I\cap J\) through the
identity map.
\end{proposition}

\begin{proof}
Consider the map
\[
\varphi:I\times J\longrightarrow B,
\qquad
\varphi(i,j)=i+j.
\]
Since \(B=I+J\), the map \(\varphi\) is surjective. We prove that \(\varphi\)
is a skew brace homomorphism. First, we show that \(\varphi\) is additive. Let
\[
i,i'\in I,
\qquad
j,j'\in J.
\]
Since \([I,J]_+=0\), we have
\[
\begin{aligned}
\varphi\bigl((i,j)+(i',j')\bigr)
&=\varphi(i+i',j+j')  \\
&=i+i'+j+j' \\
&=i+j+i'+j' \\
&=\varphi(i,j)+\varphi(i',j').
\end{aligned}
\]

We now show that \(\varphi\) is multiplicative. We have
\[
\begin{aligned}
\varphi\bigl((i,j)(i',j')\bigr)
&=\varphi(ii',jj')  \\
&=ii'+jj' \\
&=ii'jj'
    &&\bigl(ii'\in I,\ jj'\in J,\ \text{and } I*J=0\bigr)\\
&=i(i'j)j' \\
&=i(i'+j)j'
    &&\bigl(i'\in I,\ j\in J,\ \text{and } I*J=0\bigr)\\
&=i(j+i')j'
    &&\bigl([I,J]_+=0\bigr)\\
&=i(ji')j'
    &&\bigl(j\in J,\ i'\in I,\ \text{and } J*I=0\bigr)\\
&=(ij)(i'j') \\
&=(i+j)(i'+j')
    &&\bigl(i,i'\in I,\ j,j'\in J,\ \text{and } I*J=0\bigr)\\
&=\varphi(i,j)\varphi(i',j').
\end{aligned}
\]
Hence \(\varphi\) is a homomorphism of skew braces. It remains to compute its kernel. We have $(i,j)\in \ker(\varphi)$
if and only if $i+j=0.$
This is equivalent to \(i=-j\). Since \(i\in I\) and \(j\in J\), this happens
exactly when
\[
i\in I\cap J
\qquad\text{and}\qquad
j=-i.
\]
Thus
\[
\ker(\varphi)=\{(i,-i)\mid i\in I\cap J\}.
\]
This is precisely the central subgroup used to form the external central
product of \(I\) and \(J\) by amalgamating \(I\cap J\) through the identity map.
Therefore, by the first isomorphism theorem for skew braces,
\[
B\simeq (I\times J)/\{(i,-i)\mid i\in I\cap J\}.
\]
Hence \(B\) is isomorphic to the corresponding external central product.
\end{proof}
\subsection{An explicit example}

We now illustrate the notion of central product through a class of examples
coming from extraspecial \(p\)-braces. The point of the example is the
following: if the bilinear form associated with an extraspecial \(p\)-brace
admits an orthogonal decomposition, then the brace itself decomposes as a
central product of two natural ideals.
We first recall the definition of extraspecial \(p\)-brace from
\cite{BallesterBolinchesEstebanRomeroKurdachenkoPerezCalabuig2025}.

\begin{definition}
A non-abelian left brace \(E\) is an \emph{extraspecial \(p\)-brace} if
\((E,+)\) is an elementary abelian \(p\)-group and there exists an element
\(c\in Z(E)\) such that, putting $C=\langle c\rangle_+,$
the quotient \(E/C\) is a non-zero abelian left brace.
\end{definition}

The structure of extraspecial \(p\)-braces can be described in terms of
bilinear forms over \(\mathbb F_p\). We shall use the following result.

\begin{theorem}
\cite[Theorem~21]{BallesterBolinchesEstebanRomeroKurdachenkoPerezCalabuig2025}
Let \(E\) be an extraspecial \(p\)-brace, and let $C=\langle c\rangle_+\leq Z(E)$
be such that \(E/C\) is non-zero and abelian. Then the map
\[
\varphi:E/C\times E/C\longrightarrow \mathbb F_p
\]
defined by
\[
\varphi(x+C,y+C)=k_{x,y},
\qquad
x*y=k_{x,y}c,
\]
is a bilinear form. Conversely, let \(V\) be a vector space over \(\mathbb F_p\), let
\(C=\langle c\rangle_+\) be a cyclic group of order \(p\), and let
\[
\varphi:V\times V\longrightarrow \mathbb F_p
\]
be a non-zero bilinear form. Then the abelian group $E=V\oplus C$
endowed with the multiplication
\[
(x,k)(y,t)
=
(x+y,k+t+\varphi(x,y)c)
\]
is an extraspecial \(p\)-brace.
\end{theorem}

We now show how an orthogonal decomposition of the bilinear form gives rise to
a central product decomposition of the brace.

\begin{example}\rm
Let \(E\) be an extraspecial \(p\)-brace, and let $C=\langle c\rangle_+\leq Z(E)$
be such that \(E/C\) is a non-zero abelian left brace. By
\cite[Theorem~21]{BallesterBolinchesEstebanRomeroKurdachenkoPerezCalabuig2025},
the map
\[
\varphi:E/C\times E/C\longrightarrow \mathbb F_p
\]
defined by
\[
\varphi(x+C,y+C)=k_{x,y}
\qquad
\text{where}
\qquad
x*y=k_{x,y}c
\]
is a bilinear form. Suppose that \(E/C\) admits an orthogonal decomposition with respect to
\(\varphi\), say $E/C=U\oplus W,$ where $\varphi(U,W)=0=\varphi(W,U).$ Let
\[
\pi:E\longrightarrow E/C
\]
be the canonical projection, and set
\[
E_U=\pi^{-1}(U),
\qquad
E_W=\pi^{-1}(W).
\]
We claim that $E=E_U\circ E_W.$ First, \(E_U\) and \(E_W\) are ideals of \(E\). Indeed, since \(E/C\) is an
abelian left brace, the subspaces \(U\) and \(W\) are ideals of \(E/C\).
Therefore their inverse images under \(\pi\) are ideals of \(E\). Next, since $E/C=U\oplus W,$ every element of \(E\) is a sum of an element of \(E_U\) and an element of
\(E_W\). Hence $E=E_U+E_W.$
We now show that the mixed brace products vanish. Let $a\in E_U$ and $b\in E_W.$
Then $a+C\in U$ and $b+C\in W.$
By orthogonality,
\[
\varphi(a+C,b+C)=0=\varphi(b+C,a+C).
\]
By the definition of \(\varphi\), this gives
\[
a*b=0=b*a.
\]
Therefore
\[
E_U*E_W=0=E_W*E_U.
\]

Moreover, since \(E\) is an extraspecial \(p\)-brace, its additive group is
elementary abelian. In particular, \((E,+)\) is abelian, and so $[E_U,E_W]_+=0.$ Finally, we compute the intersection. Since $E/C=U\oplus W,$ we have \(U\cap W=0\). Hence
\[
E_U\cap E_W
=
\pi^{-1}(U)\cap \pi^{-1}(W)
=
\pi^{-1}(U\cap W)
=
C.
\]
Since \(C\leq Z(E)\), it follows that 
\[
E_U\cap E_W\leq Z(E).
\]

\end{example}

\section{Perfect two-sided skew braces} \label{Perfect two-sided skew braces}

In this section we study finite perfect two-sided skew braces. Our aim is to
prove the main structural decomposition announced in the introduction and to
derive its consequences for the additive and multiplicative groups. We begin by
recalling the commutator ideal of a skew brace and by introducing the notion of
perfectness used throughout the paper.

Let \(B\) be a skew brace. The \emph{commutator ideal} of \(B\), denoted by
\([B,B]^B\), is the ideal generated by the additive commutators
\([B,B]_+\), the multiplicative commutators \([B,B]_\cdot\), and the elements
\[
ab-(a+b), \qquad a,b\in B.
\]
Equivalently, \([B,B]^B\) is the smallest ideal \(I\) of \(B\) such that the
quotient \(B/I\) is an abelian skew brace. Since solvability for skew braces is
defined in terms of the commutator ideal, see
\cite{BallesterBolinchesEstebanRomeroJimenezSeralPerezCalabuig2024}, the
following definition is the natural analogue of the usual notion of a perfect
group.
\begin{definition} \label{DefinitionOfperfect}
Let $B$ be a skew brace. We say that $B$ is \emph{perfect} if $B=[B,B]^B.$
\end{definition}

This definition differs from the one given in \cite{CSV2019}, where a skew brace $B$ is called perfect if $B=B^{2}$. In general, these two notions of perfectness do not coincide. For example, if $G$ is a perfect group, then $\operatorname{Triv}(G)$ is perfect in the sense of Definition~\ref{DefinitionOfperfect}, but $\operatorname{Triv}(G) * \operatorname{Triv}(G)=0$. The goal of this section is to study perfect two-sided skew braces. Recall that a skew brace $B$ is called \emph{two-sided} if, for all $a,b,c \in B$, one has
\[
(a+b)c = ac - c + bc.
\]

\subsection{Proof of the main result}

Our first result is a characterization of almost trivial skew braces with perfect additive group, which will be a key ingredient in the proof of Theorem~\ref{ThStructure}:

\begin{lemma}\label{LemmaB*B}
Let \(B\) be a finite two-sided skew brace. Then
$B*B=B$ if and only if \(B\) is almost trivial and \((B,+)\) is a perfect group.
\end{lemma}

\begin{proof}
Suppose first that \(B\) is almost trivial and that \((B,+)\) is perfect.
Since \(B\) is almost trivial, for all \(a,b\in B\) we have
\[
a*b=-a+(b+a)-b=[-a,b]_+.
\]
Hence \(B*B=[B,B]_+\). Since \((B,+)\) is perfect, it follows that
\[
B*B=[B,B]_+=B.
\]

Conversely, assume that \(B\) is a two-sided skew brace such that $B=B*B.$
We prove first that the additive group \((B,+)\) is perfect. Let $C=[(B,+),(B,+)]_+$
be the derived subgroup of the additive group. Since \(B\) is two-sided, by
\cite[Corollary~2.4]{Trappeniers2023}, the subgroup \(C\) is an ideal of \(B\).
Therefore the quotient $\overline B=B/C$
is again a two-sided skew brace. Its additive group is abelian. Hence, by
\cite[Theorem~8.1.20]{CV}, $R=(\overline B,+,*)$
is a Jacobson radical ring, whose adjoint group is $(R,\cdot)=(\overline B,\cdot).$
Since \(B*B=B\), passing to the quotient gives
$\overline B*\overline B=\overline B.$
In the ring \(R\), this means
\[
R^2=R.
\]
Since \(B\) is finite, also \(R\) is finite, and therefore \(R\) is Artinian.
Thus, by \cite[Theorem~3.3.5]{CV}, the Jacobson radical \(J(R)\) is nilpotent.
But \(R\) is a Jacobson radical ring, so
$J(R)=R.$
Hence \(R\) is nilpotent: there exists \(n\geq 1\) such that $R^n=0.$
On the other hand, from \(R^2=R\) it follows inductively that
\[
R^m=R
\qquad
\text{for every } m\geq 1.
\]
Therefore \(\overline B=0\). Hence $B=C=[(B,+),(B,+)]_+,$ and so the additive group \((B,+)\) is perfect. It remains to prove that \(B\) is almost trivial. Since \(B\) is two-sided, by
\cite[Lemma~4.1]{Trappeniers2023} we have
\[
B^{2,\operatorname{op}}\leq C_{(B,+)}(B*B).
\]
As \(B=B*B\), this gives
\[
B^{2,\operatorname{op}}\leq C_{(B,+)}(B)=Z(B,+).
\]
Fix \(b\in B\). Consider the map
\[
\varphi_b:(B,+)\longrightarrow B^{2,\operatorname{op}},
\qquad
\varphi_b(a)=a*_{\operatorname{op}}b.
\]
We claim that \(\varphi_b\) is a group homomorphism. Let \(a,c\in B\). Since
\(B\) is two-sided, we have
\[
(a+c)b=ab-b+cb.
\]
Therefore
\[
\begin{aligned}
\varphi_b(a+c)
&=(a+c)*_{\operatorname{op}}b \\
&=-b+(a+c)b-(a+c)
    &&\bigl(\text{definition of } *_{\operatorname{op}}\bigr)\\
&=-b+ab-b+cb-c-a
    &&\bigl((a+c)b=ab-b+cb\bigr)\\
&=-b+ab-a+a-b+cb-c-a
    &&\bigl(\text{insert } -a+a\bigr)\\
&=a*_{\operatorname{op}}b+a+c*_{\operatorname{op}}b-a
    &&\bigl(\text{definition of } *_{\operatorname{op}}\bigr)\\
&=a*_{\operatorname{op}}b+c*_{\operatorname{op}}b
    &&\bigl(c*_{\operatorname{op}}b\in B^{2,\operatorname{op}}\leq Z(B,+)\bigr)\\
&=\varphi_b(a)+\varphi_b(c)
    &&\bigl(\text{definition of } \varphi_b\bigr).
\end{aligned}
\]
Thus \(\varphi_b\) is a group homomorphism.
Now \((B,+)\) is perfect, whereas $B^{2,\operatorname{op}}\leq Z(B,+)$
implies that \((B^{2,\operatorname{op}},+)\) is abelian. Therefore every group
homomorphism from \((B,+)\) to \((B^{2,\operatorname{op}},+)\) is trivial.
Hence $\varphi_b=0.$ Since \(b\in B\) was arbitrary, we get
\[
a*_{\operatorname{op}}b=0
\qquad
\text{for all } a,b\in B.
\]
Thus $B^{2,\operatorname{op}}=0,$
and \(B\) is almost trivial.
\end{proof}

\begin{lemma}\label{lemma(B,+)perfect} Let $B$ be a finite perfect two-sided skew brace. Then the additive group $(B,+)$ is perfect.
\end{lemma}

\begin{proof}
Let $C=[(B,+),(B,+)]_+$ be the derived subgroup of the additive group. Since \(B\) is two-sided, by
\cite[Corollary~2.4]{Trappeniers2023}, the subgroup \(C\) is an ideal of \(B\).
Hence we may consider the quotient skew brace
$\overline B=B/C.$ We prove that \(\overline B=0\). Since \(B\) is perfect as a skew brace, we have $B=[B,B]^B.$
Passing to the quotient gives $\overline B=[\overline B,\overline B]^{\overline B}.$
On the other hand, by construction, the additive group \((\overline B,+)\) is
abelian. Therefore all additive commutators in \(\overline B\) are trivial. We now show that every remaining generator of the commutator ideal $[\overline B,\overline B]^{\overline B}$
belongs to \(\overline B^2\). First, since \((\overline B,+)\) is abelian, for
all \(a,b\in B\) we have
\[
\begin{aligned}
\bar a*\bar b
&=
-\bar a+\bar a\cdot \bar b-\bar b \\
&=
\bar a\cdot\bar b-(\bar a+\bar b) \\
&=
\overline{a\cdot b-(a+b)}.
\end{aligned}
\]
Thus every element of the form $\overline{a\cdot b-(a+b)}$
belongs to \(\overline B^2\).
Next we prove that the multiplicative commutators also lie in
\(\overline B^2\). Indeed, modulo \(\overline B^2\), one has
\[
\bar a\cdot \bar b
=
\bar a+\bar b+\bar a*\bar b
\equiv
\bar a+\bar b
=
\bar b+\bar a
\equiv
\bar b\cdot \bar a.
\]
Hence the multiplicative group of the quotient
$\overline B/\overline B^2$ is abelian. Therefore $[\overline B,\overline B]_\cdot\leq \overline B^2.$
We have shown that all generators of $[\overline B,\overline B]^{\overline B}$
belong to \(\overline B^2\). Hence $[\overline B,\overline B]^{\overline B}\leq \overline B^2.$
Since $\overline B=[\overline B,\overline B]^{\overline B},$
we obtain $\overline B=\overline B^2.$
Now Lemma~\ref{LemmaB*B} applies to the finite two-sided skew brace
\(\overline B\). It gives that the additive group \((\overline B,+)\) is
perfect. But \((\overline B,+)\) is abelian by construction. Hence $\overline B=0.$
Therefore $B=C=[(B,+),(B,+)]_+,$
and so the additive group \((B,+)\) is perfect.
\end{proof}

\begin{lemma} \label{f_i homomorphism}
    Let $B$ be a skew brace and $I \trianglelefteq B$ with $I \le Z(B,+)$. For every $i \in I$ the map:
    \[
f_i:(B,+)\longrightarrow (I,+),
\qquad
f_i(b)=i*b.
\]
is a group homomorphism.
\end{lemma}

\begin{proof}
Since \(I\) is an ideal of \(B\), for every \(b\in B\) we have $i*b\in I.$ Thus the map $f_{i}$ is well defined.
We claim that \(f_i\) is a group homomorphism. Indeed, for \(b,c\in B\),
using the left brace identity we get $i*(b+c)=i*b+b+i*c-b.$ Since \(i*c\in I\leq Z(B,+)\) we obtain:
\[
f_i(b+c)=i*(b+c)=i*b+i*c=f_i(b)+f_i(c).
\]
Thus \(f_i\) is a group homomorphism.
\end{proof}

\begin{lemma} \label{g_i homomorphism}
     Let $B$ be a two-sided skew brace and $I \trianglelefteq B$ with $I \le Z(B,+)$. For every $i \in I$ the map:
    \[
g_i:(B,+)\longrightarrow (I,+),
\qquad
g_i(b)=b*i.
\]
is a group homomorphism.
\end{lemma}

\begin{proof}
 Let \(b,c\in B\). Since \(B\)
is two-sided and $I \le Z(B,+)$ we have:
\[
\begin{aligned}
g_{i}(b+c)
&=(b+c)*i \\
&=-(b+c)+(b+c)i-i \\
&=-c-b+bi-i+ci-i\\
&=-c+b*i+ci-i\\
&=b*i +c*i\\
&=g_{i}(b)+g_{i}(c) 
\end{aligned}
\]
Thus \(g_i\) is a group homomorphism.
\end{proof}

\begin{theorem} \label{ThStructure}
Let \(B\) be a finite perfect two-sided skew brace. Then:
\begin{enumerate}
    \item \label{point1} \(B=B^{2} \circ B^{2,\operatorname{op}}\);
    \item \label{point2} \(B^2\) is almost trivial with perfect additive group;
    \item \label{point3} \(B^{2,\operatorname{op}}\) is trivial with perfect additive group.
\end{enumerate}
\end{theorem}

\begin{proof}
We prove the theorem in several steps.

\medskip

\noindent
\textbf{Step 1: The additive decomposition.}
Put $I=B^2\cap B^{2,\operatorname{op}}.$
We first show that $B=B^2+B^{2,\operatorname{op}}.$
Indeed, the quotient $B/(B^2+B^{2,\operatorname{op}})$
is both trivial and almost trivial. Hence its additive group is abelian.
On the other hand, by Lemma~\ref{lemma(B,+)perfect}, the additive group
\((B,+)\) is perfect. Therefore every additive quotient of \(B\) which is
abelian is trivial. It follows that $B/(B^2+B^{2,\operatorname{op}})=0,$
and hence

\[
B=B^2+B^{2,\operatorname{op}}.
\]

\medskip

\noindent
\textbf{Step 2: The intersection is central.}
We now prove that
\[
I=B^2\cap B^{2,\operatorname{op}}\leq Z(B).
\]
Since \(B\) is two-sided, by \cite[Theorem~4.3]{Trappeniers2023} we have $I\leq Z(B^2+B^{2,\operatorname{op}},+).$
Using Step~1, this gives $I\leq Z(B,+).$
Let \(i\in I\). Since \(I\leq Z(B,+)\), Lemma~\ref{f_i homomorphism} implies
that the map
\[
f_i:(B,+)\longrightarrow (I,+),
\qquad
f_i(b)=i*b,
\]
is a group homomorphism. The group \((B,+)\) is perfect by
Lemma~\ref{lemma(B,+)perfect}, whereas \((I,+)\) is abelian because
\(I\leq Z(B,+)\). Hence \(f_i=0\). Thus
\[
i*b=0
\qquad
\text{for every } b\in B.
\]
Equivalently, \(\lambda_i=\operatorname{id}_B\). Since also
\(i\in Z(B,+)\), we obtain $i\in \operatorname{Soc}(B).$
It remains to show that \(i\) is central in the multiplicative group.
Again, since \(I\leq Z(B,+)\), Lemma~\ref{g_i homomorphism} implies that
\[
g_i:(B,+)\longrightarrow (I,+),
\qquad
g_i(b)=b*i,
\]
is a group homomorphism. As above, the domain is perfect and the codomain is
abelian, so \(g_i=0\). Therefore
\[
b*i=0
\qquad
\text{for every } b\in B.
\]
This implies that \(i\in Z(B,\cdot)\). Hence \(i\in Z(B)\). Since \(i\in I\)
was arbitrary, we have proved that
\[
B^2\cap B^{2,\operatorname{op}}\leq Z(B).
\]

\medskip

\noindent
\textbf{Step 3: The mixed products vanish.}
We next prove that the two summands \(B^2\) and \(B^{2,\operatorname{op}}\)
satisfy the defining relations of an internal central product. First, since \(B\) is two-sided, by \cite[Lemma~4.1]{Trappeniers2023} we have $[B^2,B^{2,\operatorname{op}}]_+=0.$
It remains to prove that
\[
B^2*B^{2,\operatorname{op}}=0
\qquad
\text{and}
\qquad
B^{2,\operatorname{op}}*B^2=0.
\]
We start with the first equality. Let \(c\in B^{2,\operatorname{op}}\). For
every \(x\in B\), we have \(x*c\in B^2\) by definition of \(B^2\). Moreover,
since \(B^{2,\operatorname{op}}\) is an ideal and \(c\in B^{2,\operatorname{op}}\),
we also have $x*c\in B^{2,\operatorname{op}}.$
Therefore $x*c\in B^2\cap B^{2,\operatorname{op}}=I.$
Thus the map
\[
\theta_c:(B,+)\longrightarrow (I,+),
\qquad
\theta_c(x)=x*c,
\]
is well defined. As in Lemma~\ref{g_i homomorphism}, the two-sidedness of \(B\)
and the inclusion \(I\leq Z(B,+)\) imply that \(\theta_c\) is a group
homomorphism. Since \((B,+)\) is perfect and \((I,+)\) is abelian, we obtain $\theta_c=0.$
Hence
\[
x*c=0
\qquad
\text{for every } x\in B.
\]
In particular, $B^2*B^{2,\operatorname{op}}=0.$
We now prove the opposite mixed product. Consider the opposite skew brace
$C=B^{\operatorname{op}}.$
Then \(C\) is again a finite perfect two-sided skew brace, and
\[
C^2=B^{2,\operatorname{op}},
\qquad
C^{2,\operatorname{op}}=B^2.
\]
Applying the previous argument to \(C\), we obtain $C^2*C^{2,\operatorname{op}}=0.$
In terms of \(B\), this says that $B^{2,\operatorname{op}}*_{\operatorname{op}}B^2=0.$ Let \(a\in B^{2,\operatorname{op}}\) and \(b\in B^2\). Since $a*_{\operatorname{op}}b=0,$
we have \(ab=b+a\). Therefore
\[
\begin{aligned}
a*b
&=-a+ab-b \\
&=-a+(b+a)-b
    &&\bigl(a*_{\operatorname{op}}b=0,\ \text{so } ab=b+a\bigr)\\
&=-a+(a+b)-b
    &&\bigl([B^2,B^{2,\operatorname{op}}]_+=0,\ \text{so } b+a=a+b\bigr)\\
&=0.
\end{aligned}
\]
Thus $B^{2,\operatorname{op}}*B^2=0.$ Combining Step~1, Step~2, and the equalities just proved, we conclude that $B=B^2\circ B^{2,\operatorname{op}}.$
This proves item~\ref{point1}.

\medskip

\noindent
\textbf{Step 4: The component \(B^2\).}
We now prove item~\ref{point2}. In Step~3 we proved that, for every
\(c\in B^{2,\operatorname{op}}\),
\[
x*c=0
\qquad
\text{for every } x\in B.
\]
Equivalently,
\[
B*B^{2,\operatorname{op}}=0.
\]

Let \(x,y\in B\). By Step~1, we may write
\[
x=c+a,
\qquad
y=c'+a',
\]
with
\[
c,c'\in B^2,
\qquad
a,a'\in B^{2,\operatorname{op}}.
\]
Recall that, for elements of the additive group, we write
\[
u^v=-v+u+v.
\]
Using the left brace identity, the two-sided identity, and the mixed product
relations proved above, we compute:
\[
\begin{aligned}
x*y
&=(c+a)*(c'+a') \\
&=(c+a)*c' + c' + (c+a)*a' - c'
    &&\bigl(\text{left brace identity}\bigr)\\
&=(c+a)*c'
    &&\bigl(B*B^{2,\operatorname{op}}=0\bigr)\\
&=(c*c')^a+a*c'
    &&\bigl(\text{two-sided identity}\bigr)\\
&=(c*c')^a
    &&\bigl(B^{2,\operatorname{op}}*B^2=0\bigr)\\
&=c*c'
    &&\bigl([B^2,B^{2,\operatorname{op}}]_+=0\bigr).
\end{aligned}
\]
Thus every product \(x*y\) belongs to \(B^2*B^2\). Since \(B^2\) is generated
additively by all such elements \(x*y\), we get
$B^2\leq B^2*B^2.$ The reverse inclusion is immediate, so $B^2=B^2*B^2.$
Applying Lemma~\ref{LemmaB*B} to the finite two-sided skew brace \(B^2\), we
deduce that \(B^2\) is almost trivial and that its additive group \((B^2,+)\)
is perfect. This proves item~\ref{point2}.

\medskip

\noindent
\textbf{Step 5: The component \(B^{2,\operatorname{op}}\).}
Finally, we prove item~\ref{point3}. Apply the argument of Step~4 to the
opposite skew brace \(B^{\operatorname{op}}\). Since \(B^{\operatorname{op}}\)
is again a finite perfect two-sided skew brace and $(B^{\operatorname{op}})^2=B^{2,\operatorname{op}},$
we obtain
\[
B^{2,\operatorname{op}}
=
B^{2,\operatorname{op}}*_{\operatorname{op}}B^{2,\operatorname{op}}.
\]
By Lemma~\ref{LemmaB*B}, applied inside \(B^{\operatorname{op}}\), it follows
that the additive group $(B^{2,\operatorname{op}},+)$
is perfect. On the other hand, Step~3 gives
$B*B^{2,\operatorname{op}}=0.$ In particular,
\[
B^{2,\operatorname{op}}*B^{2,\operatorname{op}}=0.
\]
Therefore \(B^{2,\operatorname{op}}\) is trivial. This proves
item~\ref{point3}, and the theorem follows.
\end{proof}

\subsection{Some consequences of Theorem \ref{ThStructure}}

As a first consequence, for a two-sided skew brace, perfectness of the skew brace is equivalent to perfectness of both of its underlying groups.

\begin{corollary}\label{CorollaryEquivalentconditions}
Let \(B\) be a finite two-sided skew brace. Then the following conditions are equivalent:
\begin{enumerate}
    \item \(B\) is perfect;
    \item the additive group \((B,+)\) is perfect;
    \item the multiplicative group \((B,\cdot)\) is perfect.
\end{enumerate}
\end{corollary}

\begin{proof}
If \((B,+)\) is perfect, then \([B,B]_+=B\), and hence
\(B=[B,B]^B\). Similarly, if \((B,\cdot)\) is perfect, then
\([B,B]_{\cdot}=B\), and again \(B=[B,B]^B\).

It remains to prove that, if \(B\) is perfect as a skew brace, then its
multiplicative group \((B,\cdot)\) is perfect. Assume therefore that \(B\)
is perfect. By Theorem~\ref{ThStructure}, we have
\[
B \simeq
\bigl(\operatorname{Triv}(G)\times \operatorname{aTriv}(H)\bigr)/I
\]
for some finite perfect groups \(G\) and \(H\), and for some central ideal
\(I\). Hence \((B,\cdot)\) is a quotient of the perfect group
\((G,\cdot)\times (H,\cdot)\). It follows that \((B,\cdot)\) is perfect.
\end{proof}

The following example shows that the two-sidedness assumption in
Corollary~\ref{CorollaryEquivalentconditions} is essential.

\begin{example}\rm \label{oddexample}
By \cite[Theorem 1.1]{Byott2026}, there exists a simple skew brace $S$ of order
\[
221875=5^{5}\cdot 71
\]
such that
\[
(S,+)\simeq (C_{5})^{5}\rtimes C_{71}
\qquad\text{and}\qquad
(S,\cdot)\simeq C_{71}\rtimes P,
\]
where $P$ is a group of order $5^{5}$ acting non-trivially on $C_{71}$. Since $S$ is non-trivial and simple, it follows that $S=[S,S]^S$, and hence $S$ is perfect. However, neither $(S,+)$ nor $(S,\cdot)$ is a perfect group.
\end{example}

The following corollary generalizes the classification of finite simple two-sided skew braces:

\begin{corollary} \label{CorollaryZ(B)=0}
Let $B$ be a finite perfect two-sided skew brace. If $Z(B)=0$, then
\[
B \simeq \operatorname{Triv}(G) \times \operatorname{aTriv}(H)
\]
for some finite perfect groups $G$ and $H$.
\end{corollary}

\begin{proof}
By Theorem~\ref{ThStructure}, we have
\[
B \simeq \bigl(\operatorname{Triv}(G)\times \operatorname{aTriv}(H)\bigr)/I
\]
for some finite perfect groups \(G\) and \(H\), and for some central ideal \(I\). If $Z(B)=0$, then $I=0$, and therefore
\[
B \simeq \operatorname{Triv}(G)\times \operatorname{aTriv}(H).
\]
\end{proof}

\begin{corollary}
Let \(B\) be a finite perfect two-sided skew brace. Then
\[
B/Z(B)\simeq \operatorname{Triv}(G)\times \operatorname{aTriv}(H)
\]
for some finite perfect groups \(G\) and \(H\).
\end{corollary}

\begin{proof}
Since \(B\) is a finite perfect two-sided skew brace, by
\cite[Corollary~1.8]{Tsang2026} we have \(Z(B/Z(B))=0\). Therefore, applying
Corollary~\ref{CorollaryZ(B)=0} to \(B/Z(B)\), we obtain
\[
B/Z(B)\simeq \operatorname{Triv}(G)\times \operatorname{aTriv}(H)
\]
for some finite perfect groups \(G\) and \(H\).
\end{proof}

\begin{definition}  
   We say that a skew brace $B$ is quasi-simple if $B$ is perfect and $B/Z(B)$ is simple.
\end{definition}

 In the following we apply Theorem~\ref{ThStructure} to the classification of quasi-simple two-sided skew braces. 

\begin{corollary}
Let $B$ be a finite two-sided skew brace. If one of the following conditions
holds:
\begin{enumerate}
    \item \(B\) is quasi-simple;
    \item the additive group \((B,+)\) is quasi-simple;
    \item the multiplicative group \((B,\cdot)\) is quasi-simple,
\end{enumerate}
then $B$ is either trivial or almost trivial.
\end{corollary}

\begin{proof}
We prove the three cases separately.

\medskip

\noindent
\textbf{Case 1: \(B\) is quasi-simple as a skew brace.}
Assume first that \(B\) is quasi-simple. Thus \(B\) is perfect and
\(B/Z(B)\) is a simple skew brace. Since \(B\) is finite, perfect and two-sided, Theorem~\ref{ThStructure} gives $B=B^2\circ B^{2,\operatorname{op}}.$ Put $I=B^2$ and $J=B^{2,\operatorname{op}}.$ By the definition of internal central product, we have
\[
B=I+J,\qquad [I,J]_+=0,\qquad I*J=J*I=0,
\qquad I\cap J\leq Z(B).
\]
Moreover, by Theorem~\ref{ThStructure}, \(I\) is almost trivial with perfect
additive group, while \(J\) is trivial with perfect additive group. Let $\overline B=B/Z(B),$
and denote by \(\overline I\) and \(\overline J\) the images of \(I\) and \(J\)
in \(\overline B\). Then \(\overline I\) and \(\overline J\) are ideals of
\(\overline B\), and
\[
\overline B=\overline I+\overline J.
\]
Since \(\overline B\) is simple, each of \(\overline I\) and \(\overline J\) is
either zero or equal to \(\overline B\). We claim that \(\overline I\) and \(\overline J\) cannot both be equal to
\(\overline B\). Indeed, if $\overline I=\overline B=\overline J,$
then the relations \([I,J]_+=0\) and \(I*J=0\) imply
\[
[\overline B,\overline B]_+=0
\qquad\text{and}\qquad
\overline B*\overline B=0.
\]
Hence \(\overline B\) is abelian. On the other hand, since \(B\) is perfect,
also \(\overline B\) is perfect. Therefore \(\overline B=0\), a contradiction.
Thus one of \(\overline I\) and \(\overline J\) is zero. Suppose first that \(\overline I=0\). Then
\[
I=B^2\leq Z(B).
\]
By Theorem~\ref{ThStructure}, the additive group \((I,+)\) is perfect. But
\(I\leq Z(B)\leq Z(B,+)\), so \((I,+)\) is abelian. Hence \(I=0\). Therefore $B^2=0,$
and \(B\) is trivial. Suppose now that \(\overline J=0\). Then
\[
J=B^{2,\operatorname{op}}\leq Z(B).
\]
Again, by Theorem~\ref{ThStructure}, the additive group \((J,+)\) is perfect,
while \(J\leq Z(B,+)\) implies that \((J,+)\) is abelian. Hence \(J=0\).
Therefore $B^{2,\operatorname{op}}=0,$
and \(B\) is almost trivial. This proves the result when \(B\) is quasi-simple as a skew brace.

\medskip

\noindent
\textbf{Case 2: The additive group \((B,+)\) is quasi-simple.}
Assume now that the additive group \((B,+)\) is quasi-simple. In particular,
\((B,+)\) is perfect. Hence \(B\) is perfect as a skew brace. Therefore, by
Theorem~\ref{ThStructure}, $B=B^2\circ B^{2,\operatorname{op}}.$
Since \(B\) is two-sided and \(Z(B,+)\) is characteristic in \((B,+)\), by
\cite[Corollary~2.4]{Trappeniers2023} the subgroup \(Z(B,+)\) is an ideal of
\(B\). Put $\widetilde B=B/Z(B,+).$
Since \((B,+)\) is quasi-simple, the group \((\widetilde B,+)\) is simple.
Moreover, \(\widetilde B\) is again finite, perfect and two-sided. Applying Theorem~\ref{ThStructure} to \(\widetilde B\), we get $\widetilde B
=
\widetilde B^2\circ \widetilde B^{2,\operatorname{op}}.$
In particular,
\[
\widetilde B^2\cap \widetilde B^{2,\operatorname{op}}
\leq Z(\widetilde B)
\leq Z(\widetilde B,+).
\]
Since \((\widetilde B,+)\) is simple, either
\[
Z(\widetilde B,+)=0
\qquad\text{or}\qquad
Z(\widetilde B,+)=\widetilde B.
\]
The second possibility cannot occur. Indeed, if
\(Z(\widetilde B,+)=\widetilde B\), then \((\widetilde B,+)\) is abelian.
Since \(\widetilde B\) is perfect, this would force \(\widetilde B=0\), which
is impossible. Hence $Z(\widetilde B,+)=0.$
Therefore $\widetilde B^2\cap \widetilde B^{2,\operatorname{op}}=0.$
It follows that
\[
(\widetilde B,+)
\simeq
(\widetilde B^2,+)\times
(\widetilde B^{2,\operatorname{op}},+).
\]
Since \((\widetilde B,+)\) is simple, one of the two factors must be zero.
Thus
\[
\widetilde B^2=0
\qquad\text{or}\qquad
\widetilde B^{2,\operatorname{op}}=0.
\]
Suppose first that $\widetilde B^2=0.$
Then $B^2\leq Z(B,+).$
Since \(B=B^2+B^{2,\operatorname{op}}\), we get
$B=Z(B,+)+B^{2,\operatorname{op}}.$
Therefore
\[
(B,+)/B^{2,\operatorname{op}}
\simeq
Z(B,+)/(Z(B,+)\cap B^{2,\operatorname{op}}).
\]
The right-hand side is abelian. Hence \((B,+)/B^{2,\operatorname{op}}\) is
abelian. Since \((B,+)\) is perfect, this quotient is trivial. Therefore $B=B^{2,\operatorname{op}}.$
Equivalently, $B^{\operatorname{op}}=(B^{\operatorname{op}})^2.$
By Lemma~\ref{LemmaB*B}, applied to \(B^{\operatorname{op}}\), the skew brace
\(B^{\operatorname{op}}\) is almost trivial. Hence \(B\) is trivial.
The case $\widetilde B^{2,\operatorname{op}}=0$
is analogous and gives that \(B\) is almost trivial.

\medskip

\noindent
\textbf{Case 3: The multiplicative group \((B,\cdot)\) is quasi-simple.}
Finally, assume that the multiplicative group \((B,\cdot)\) is quasi-simple.
By \cite[Theorem~1.4]{TsangQuasisimple2021}, for finite two-sided skew braces,
if \((B,\cdot)\) is quasi-simple, then
$(B,+)\simeq (B,\cdot).$
Hence \((B,+)\) is quasi-simple. Therefore Case~2 applies, and \(B\) is either
trivial or almost trivial.
\end{proof}

The next example shows that in the non-two-sided setting quasi-simplicity does not imply that the skew brace is trivial. Let \((G,\cdot)\) be a group, and let \(H,K\leq G\). We say that \(G\) admits an exact factorization through \(H\) and \(K\) if
\[
G=HK \qquad \text{and} \qquad H\cap K=1.
\]
Equivalently, every element \(x\in G\) can be written uniquely as \(x=hk\), with \(h\in H\) and \(k\in K\). In this situation, by \cite[Theorem~2.3]{SmoktunowiczVendramin2018}, one obtains a skew brace \((G,\cdot,\circ)\) as follows. For \(x,y\in G\), write \(x=hk\) with \(h\in H\) and \(k\in K\). Then define
\[
x\circ y = h y k.
\]

\begin{example} \rm
Recall that \(A_5\) admits the exact factorization $A_5=A_4C_5,$
where \(A_4\) is the stabilizer of a point in the natural action of \(A_5\) on
five points, and \(C_5\) is generated by a \(5\)-cycle. Thus $A_4\cap C_5=1,$
and every element of \(A_5\) can be written uniquely as a product \(xy\), with
\(x\in A_4\) and \(y\in C_5\).
By \cite[Theorem~2.3]{SmoktunowiczVendramin2018}, this exact factorization
defines a skew brace \(S\) whose additive group is \(A_5\) and whose
multiplicative group is isomorphic to
$A_4\times C_5.$ We now lift this construction to \(\operatorname{SL}(2,5)\). Let
\[
\pi:\operatorname{SL}(2,5)\longrightarrow
\operatorname{SL}(2,5)/Z(\operatorname{SL}(2,5))\simeq A_5
\]
be the natural quotient map. Write $A_5=XY,$
where \(X\simeq A_4\) and \(Y\simeq C_5\). Put
$\widetilde X=\pi^{-1}(X).$ Then $|\widetilde X|=2|X|=24.$ Moreover, \(\pi^{-1}(Y)\) has order \(10\). Choose a subgroup $\widetilde H\leq \pi^{-1}(Y)$
of order \(5\). Then \(\pi\) maps \(\widetilde H\) isomorphically onto \(Y\). We claim that
$\operatorname{SL}(2,5)=\widetilde X\widetilde H$
is an exact factorization. Indeed,
\[
|\widetilde X|\,|\widetilde H|
=
24\cdot 5
=
120
=
|\operatorname{SL}(2,5)|.
\]
Moreover, since \(24\) and \(5\) are coprime, we have $\widetilde X\cap \widetilde H=1.$
Thus the factorization is exact. Again by \cite[Theorem~2.3]{SmoktunowiczVendramin2018}, this exact
factorization defines a skew brace \(\widetilde S\) whose additive group is
$(\widetilde S,+)=\operatorname{SL}(2,5)$
and whose multiplicative group is isomorphic to
$(\widetilde S,\circ)\simeq \widetilde X\times \widetilde H.$
Since \(\operatorname{SL}(2,5)\) is perfect, the skew brace \(\widetilde S\)
is perfect. We next prove that
\[
\widetilde S/Z(\widetilde S)\simeq S.
\]
First, we show that \(\pi:\widetilde S\to S\) is a homomorphism of skew braces.
It is clearly a homomorphism for the additive groups. We check that it is also
a homomorphism for the multiplicative operations. Let \(g,t\in \widetilde S\), and write
\[
g=xh,
\qquad
x\in \widetilde X,\quad h\in \widetilde H,
\]
By the definition of the
multiplicative operation associated with an exact factorization, we have
\[
g\circ t=xth.
\]
Therefore
\[
\pi(g\circ t)
=
\pi(xth)
=
\pi(x)\pi(t)\pi(h).
\]
Since
\[
\pi(g)=\pi(x)\pi(h)
\]
is the decomposition of \(\pi(g)\) with respect to the exact factorization
\(A_5=XY\), the definition of the multiplicative operation on \(S\) gives
\[
\pi(g)\circ \pi(t)
=
\pi(x)\pi(t)\pi(h).
\]
Hence
\[
\pi(g\circ t)=\pi(g)\circ \pi(t),
\]
and so \(\pi:\widetilde S\to S\) is a skew brace homomorphism. Its kernel is $\ker(\pi)=Z(\operatorname{SL}(2,5)).$
It remains to prove that $Z(\widetilde S)=Z(\operatorname{SL}(2,5)).$
By definition of the center of a skew brace, every element of \(Z(\widetilde S)\)
is central in the additive group. Since the additive group of \(\widetilde S\)
is \(\operatorname{SL}(2,5)\), we immediately get $Z(\widetilde S)\leq Z(\operatorname{SL}(2,5)).$
Conversely, let $z\in Z(\operatorname{SL}(2,5)).$ Since
\[
\ker(\pi)=Z(\operatorname{SL}(2,5))\leq \widetilde X,
\]
the decomposition of \(z\) with respect to the exact factorization $\operatorname{SL}(2,5)=\widetilde X\widetilde H$
is simply
\[
z=z\cdot 1,
\qquad
z\in \widetilde X,\quad 1\in \widetilde H.
\]
We first show that \(z\in \ker(\lambda)\). For every \(t\in \widetilde S\),
using the definition of the operation \(\circ\), we have
\[
z\circ t=zt.
\]
Therefore
\[
\lambda_z(t)
=
z^{-1}(z\circ t)
=
z^{-1}zt
=
t.
\]
Thus
\[
\lambda_z=\operatorname{id}.
\]

It remains to show that \(z\) is central in the multiplicative group
\((\widetilde S,\circ)\). Let \(g\in \widetilde S\), and write
\[
g=xh,
\qquad
x\in \widetilde X,\quad h\in \widetilde H.
\]
Then
\[
\begin{aligned}
g\circ z
&=xzh
    &&\bigl(\text{by the definition of } \circ\bigr)\\
&=xhz
    &&\bigl(z\in Z(\operatorname{SL}(2,5))\bigr)\\
&=gz
    &&\bigl(g=xh\bigr)\\
&=zg
    &&\bigl(z\in Z(\operatorname{SL}(2,5))\bigr)\\
&=z\circ g
    &&\bigl(z=z\cdot 1,\ z\in\widetilde X,\ 1\in\widetilde H\bigr).
\end{aligned}
\]
Therefore $z\in Z(\widetilde S,\circ).$
Together with \(\lambda_z=\operatorname{id}\) and
\(z\in Z(\operatorname{SL}(2,5))=Z(\widetilde S,+)\), this gives $z\in Z(\widetilde S).$
Hence $Z(\operatorname{SL}(2,5))\leq Z(\widetilde S).$
We conclude that
\[
Z(\widetilde S)=Z(\operatorname{SL}(2,5)).
\]
By the first isomorphism theorem for skew braces we obtain
\[
\widetilde S/Z(\widetilde S)\simeq S.
\]
As \(S\) is simple and \(\widetilde S\) is perfect, it follows that
\(\widetilde S\) is quasi-simple.

\end{example}

\end{document}